\documentclass[12pt,reqno,a4paper]{amsart}

\usepackage{amsmath,amssymb,amsthm,amsfonts} 
\usepackage{mathtools} 
\usepackage{thmtools} 
\usepackage{thm-restate} 
\usepackage{enumitem} 
\usepackage{tikz-cd} 
\usepackage{xcolor} 
\usepackage{hyperref} 
\usepackage[capitalise]{cleveref} 
\usepackage{comment} 
\usepackage{xcolor} 
\usepackage{soul} 
\usepackage{MnSymbol} 

\numberwithin{equation}{subsection}

\newcommand{\PSL}{\operatorname{PSL}}

\newcommand{\Aut}{\operatorname{Aut}}
\newcommand{\Inn}{\operatorname{Inn}}
\newcommand{\Out}{\operatorname{Out}}

\newcommand{\trivgp}{\operatorname{\langle e \rangle}}
\newcommand{\abs}[1]{\left| #1 \right|}
\newcommand{\norm}[1]{\left \Vert #1 \right \Vert}

\DeclareMathOperator{\Stab}{Stab}

\DeclareMathOperator{\Fix}{Fix}

\DeclareMathOperator{\Env}{Env}
\DeclareMathOperator{\Prob}{Prob}
\DeclareMathOperator{\IRS}{IRS}
\DeclareMathOperator{\Sub}{Sub}

\DeclareMathOperator{\Alt}{Alt}
\DeclareMathOperator{\Mod}{Mod}
\DeclareMathOperator{\Core }{Core}

\newcommand{\N}{\mathbb N}
\newcommand{\Z}{\mathbb Z}

\newcommand{\cB}{\mathcal{B}}

\def\acts{\curvearrowright}
\newcommand\subsetsim{\mathrel{%
\ooalign{\raise0.2ex\hbox{$\subset$}\cr\hidewidth\raise-0.8ex\hbox{\scalebox{0.9}{$\sim$}}\hidewidth\cr}}}

\newtheorem{theorem}{Theorem}[section]

\newtheorem{proposition}[theorem]{Proposition}
\newtheorem{lemma}[theorem]{Lemma}

\newtheorem{question}{Question}

\theoremstyle{definition}
\newtheorem{definition}[theorem]{Definition}

\newtheorem{remark}[theorem]{Remark}
\newtheorem{example}[theorem]{Example}

\theoremstyle{remark}

\allowdisplaybreaks

\newlist{enumthm}{enumerate}{1} 
\setlist[enumthm]{label=\upshape(\alph*),ref=\upshape\thetheorem(\alph*)}
\crefalias{enumthmi}{theorem} 

\newlist{enumprop}{enumerate}{1} 
\setlist[enumprop]{label=\upshape(\alph*),ref=\upshape\theproposition(\alph*)}
\crefalias{enumpropi}{proposition} 

\newlist{enumlem}{enumerate}{1} 
\setlist[enumlem]{label=\upshape(\alph*),ref=\upshape\thelemma(\alph*)}
\crefalias{enumlemi}{lemma} 

\newlist{enumcor}{enumerate}{1} 
\setlist[enumcor]{label=\upshape(\alph*),ref=\upshape\thecorollary(\alph*)}
\crefalias{enumcori}{corollary} 

\newlist{enumdef}{enumerate}{1} 
\setlist[enumdef]{label=\upshape(\alph*),ref=\upshape\thedefinition(\alph*)}
\crefalias{enumdefi}{definition} 

\title[Faithful invariant random subgroups]{Faithful invariant random subgroups in acylindrically hyperbolic groups}

\author{Yair Glasner} 
\address{Yair Glasner, Ben Gurion University of the Negev.
	Departement of Mathematics.
	Be'er Sheva, 8410501, Israel.
}

\author{Anton Hase} 
\address{Anton Hase, Ben Gurion University of the Negev.
	Departement of Mathematics.
	Be'er Sheva, 8410501, Israel.
}

\date{\today}

\begin{document}

\maketitle

\begin{abstract}
Building on work from Sun and Kechris-Quorning, we prove that every acylindrically hyperbolic group $G$ admits a weakly mixing probability measure preserving action $G \curvearrowright (X,\cB,\mu)$ which is faithful but not essentially free. 
In other words, $G$ admits a weakly mixing nontrivial faithful IRS. We also prove that every non-elementary hyperbolic group admits a characteristic random subgroup with the same properties.
\end{abstract}

\section{Introduction}\label{sec:intro}
Invariant random subgroups (IRSs) is a generalization of normal subgroups, that attracts more and more interest since they were introduced in \cite{Abert_Glasner_Virag_2014}. While normal subgroups are trivial examples of IRSs, they play a nontrivial role in the theory. Every IRS $\mu$ is intimately connected to two normal subgroups. Inscribed in $\mu$ lies its \emph{kernel} $\ker(\mu)$, the maximal normal subgroup that is almost surely contained in the IRS. The \emph{normal closure} $\langle\mu\rangle$ circumscribes the IRS; it is the minimal normal subgroup that almost surely contains the IRS.
The leading question of our research is the relation between those normal subgroups. Which groups $G$ admit IRSs that have 'small' kernels, but 'big' normal closures? To be more precise if ${\ker(\mu)=\trivgp}$ we say that $\mu$ is \emph{faithful}, if ${\langle \mu \rangle = \trivgp}$ we say that $\mu$ is \emph{trivial} and if ${\langle \mu \rangle = G}$ we call $\mu$ \emph{spanning}. So we can ask: 

\begin{question}\label{q:1}
Which (countable) groups admit ergodic nontrivial faithful IRSs? Can these IRSs even be spanning? 
\end{question} If we do not demand ergodicity, the problem vanishes: A convex combination of a nontrivial IRS with the trivial IRS $\delta_{\trivgp}$ is nontrivial and faithful. 

At first glance a promising area to look for interesting IRSs are groups which have only few normal subgroups. Every IRS in a simple group $G$, that is not $\delta_{\trivgp}$ or $\delta_G$, is faithful and spanning. This applies for example to the nontrivial IRSs in the finitary alternating group $\Alt^{\infty}$ constructed by Vershik in \cite{Ver:tnf}.  However, perhaps surprisingly, this method does not go very far. 
A sequence of deep works shows that groups that are deficient in normal subgroups also tend to admit very few IRSs. A center-free lattice in a higher rank simple Lie group, for example, is just infinite by Margulis' normal subgroup theorem. The Nevo-Stuck-Zimmer theorem \cite{SZ, NZ} guarantees in this case that every nontrivial IRS too is of finite index and hence cannot be faithful. A similar result holds for lattices in Lie groups over local fields \cite{Levit:NZ}. The group $\PSL_n(K)$ for a countable field $K$ is simple, and here the work of Peterson and Thom \cite{PT:char_rig} shows that $G=\PSL_n(K)$ admits no other IRSs than $\delta_{\trivgp}$ or $\delta_G$. A similar situation holds for Thompson's groups $F$, $V$ and $T$ by the work of Le Boudec and Matte Bon \cite{BM:Thompson}. These results are probably the most striking manifestations of the deep relation between IRSs and normal subgroups alluded to earlier.

On the other hand stand groups that are rich in normal subgroups, such as Gromov hyperbolic groups or, more generally, acylindrically hyperbolic groups. These are often rich in IRSs, beyond those coming from the normal subgroups themselves. For example \cite{Bowen_2015}, \cite{Eisenmann_Glasner_2016}, \cite{Bowen_Grigorchuk_Kravchenko_2015}, \cite{Hartman_Yadin_2018}, \cite{Joseph_2021} and many other papers construct interesting IRSs in free groups, lamplighter groups and surface groups. It can be easily verified that many of these IRSs are faithful. In free and surface groups, many of them are also spanning. The work of Bowen-Grigorchuk-Kravchenko  \cite{Bowen_Grigorchuk_Kravchenko_2017} addresses a much larger class of groups, including all acylindrically hyperbolic groups. For this they construct characteristic random subgroups inside the countably generated free group $F_{\infty}$. These are IRSs whose law is invariant under all automorphisms, and as such, they give rise to an IRS inside any group that contains $F_{\infty}$ as a normal subgroup. These IRSs are even weakly mixing. However they are very far from being faithful.

To overcome this problem, we will use a different method of (co)-inducing IRSs from subgroups to the ambient group. This method was introduced and applied to many groups by Kechris-Quorning \cite{Kechris_Quorning_2019}. Building on their work and results of \cite{Sun_2020} on acylindrically hyperbolic group, we prove the following:
\begin{restatable*}{theorem}{acylin}\label{thm:acylin}
Every acylindrically hyperbolic group admits a weakly mixing nontrivial faithful IRS.
\end{restatable*}
Kechris and Quorning use the same method in order to construct a faithful characteristic IRS inside the free group $F_2$. We generalize this result too in the following: 
\begin{restatable*}{theorem}{charirs}\label{thm:charirs} 
Every non-elementary hyperbolic group admits a weakly mixing nontrivial faithful characteristic random subgroup.
\end{restatable*}

\section{Background on IRSs}\label{sec:background}
\subsection{Kernel and normal closure of an IRS}
Even though our main theorems deal only with countable groups, we develop the terminology in a slightly more general setting. Thus, let $G$ be a locally compact second countable group and let $\Sub(G)$ be the space of closed subgroups of $G$ equipped with the Chabauty topology. The Chabauty topology is generated by the sets \begin{align*}
O_1(U)&\coloneqq\{H\in \Sub(G)\mid H\cap U\neq\emptyset\},\\
O_2(K)&\coloneqq\{H\in \Sub(G)\mid H\cap K=\emptyset\},
\end{align*} for $U\subset G$ open and $K\subset G$ compact. When the group $G$ is a discrete countable group, the Chabauty topology coincides with the topology induced on $\Sub(G)$ from the natural Tychonoff topology on $2^G$. $\Sub(G)$ endowed with the Chabauty topology is a compact metrizable space, and we will often refer without further notice to some metric compatible with the topology. In particular, both $G$ and $\Sub(G)$ are hereditarily Lindel\"of. Namely, every open cover of a subspace admits a countable subcover. We will use this property several times. 

\begin{definition}\label{def:IRS}
An \emph{invariant random subgroup (IRS)} is a $G$-invariant Borel probability measure on $\Sub(G)$. We denote the set of IRSs on $G$ by $\IRS(G)$.
\end{definition}

Every IRS $\mu\in \IRS(G)$ is induced by a probability measure preserving action $G\acts(\Omega,\nu)$ via the stabilizer map, i.e., $\mu=\Stab_*\nu$ (see Proposition 13 in \cite{Abert_Glasner_Virag_2014} and Theorem 2.6 in \cite{7samurai_2017}). Even though we are only interested in IRSs, we formulate many statements in terms of probability measures on $\Sub(G)$. This does not complicate the proofs and allows us to talk comfortably about IRSs on subgroups.

\subsubsection*{The kernel of an IRS}
For any subset $F\subset G$, the \emph{envelope} of $F$ is defined as the collection of all subgroups containing $F$, \[\Env F\coloneqq\{H\in \Sub(G)\mid F\subseteq H\}.\] For a singleton $g\in G$, we write $\Env g$ instead of $\Env  \{g\}$. Note that $\Env F=\bigcap_{g\in F} \Env g=\bigcap_{g\in F} \Sub(G)\setminus O_2(\{g\})$ is closed.

\begin{definition}\label{def:kernel}
Let $\mu\in \Prob(\Sub(G))$. The \emph{kernel} of $\mu$ is \begin{equation*}
\ker(\mu)\coloneqq \max_{H\in \Sub(G)} \{H\mid \mu(\Env H)=1\}.
\end{equation*}
\end{definition}

\begin{lemma}\label{lem:kernel}
For $\mu \in \Prob(\Sub(G))$, the kernel $\ker(\mu)$ exists and is unique. Moreover $\ker(\mu)=\{g\in G\mid \mu(\Env g)=1\}.$
\end{lemma}

\begin{proof}
Let $W = \{H \in \Sub(G) \ \mid \ \mu(\Env H)=1\}$ and $E=\overline{\langle \cup_{H \in W} H \rangle} \in \Sub(G)$. It is enough to show that $\mu(\Env E)=1$. Note that $\Env E=\bigcap_{H \in W} \Env H$, so if $\Delta \not \in \Env E$ then $\Delta \not \in \Env H_{\Delta}$ for some $H_\Delta \in W$. Since $\Env H_\Delta$ is closed, there is $\epsilon_\Delta>0$ such that $\Env H_\Delta \cap B(\Delta,\epsilon_\Delta)=\emptyset$. Since $\Sub(G)\setminus \Env E$ is Lindel\"of, its cover $\bigcup_{\Delta \not \in \Env E} B(\Delta,\epsilon_\Delta)$ admits a countable subcover $\bigcup_{n\in \N} B(\Delta_n,\epsilon_{\Delta_n})$. It follows that $\mu(\Env E)=\mu(\bigcap_{n\in \N} \Env H_{\Delta_n})=1$. 

For any $g \in \ker(\mu)$, we have $\mu(\Env g) \geq \mu(\Env \ker(\mu)) = 1$. Conversely, if $\mu(\Env g)=1$ then the cyclic group $\langle g \rangle \in W$ and consequently $g \in \ker \mu$.
\end{proof}

\begin{definition}
A probability measure $\mu\in \Prob(\Sub(G))$ is \emph{faithful} if ${\ker(\mu)=\langle e\rangle}$.
\end{definition}

If $\mu\in \IRS(G)$, then $\ker(\mu)$ is normal. If $\mu=\Stab_*\nu$, the kernel of $\mu$ is the kernel of the action $G\acts(\Omega,\nu)$. The IRS $\mu$ is faithful if it is induced by a faithful action.

\subsubsection*{The normal closure of an IRS}

The other subgroup we are interested in is the normal closure of an IRS. This group was defined in \cite{Hartman_Tamuz_2016}.

\begin{definition}\label{def:closure}
Let $\mu\in \Prob(\Sub(G))$. The \emph{closure} of $\mu$ is \begin{equation*}
\langle \mu \rangle\coloneqq \min_{H\in \Sub(G)} \{H\mid \mu(\Sub(H))=1\}.
\end{equation*} If $\mu\in \IRS(G)$, then $\langle \mu \rangle$ is called the \emph{normal closure} of $\mu$.
\end{definition}

Existence and uniqueness of the closure can be shown exactly as for the kernel. We include the proof for the sake of completeness. 
\begin{lemma}\label{lem:closure}
For $\mu \in \Prob(\Sub(G))$, the closure $\langle \mu \rangle$ exists and is unique.
\end{lemma}

\begin{proof} 
Let us denote $W=\{H \in \Sub(G) \mid \mu(\Sub(H)) = 1 \}$ and $K=\bigcap_{H \in W} H$. It is enough to show that $\mu(\Sub(K))=1$. Note that $\Sub(K)=\bigcap_{H \in W} \Sub(H)$. For $\Delta\in \Sub(G)\setminus \Sub(K)$, there is $H_\Delta\in W$ such that $\Delta \notin \Sub(H_\Delta)$. Since $\Sub(H_\Delta)$ is closed, there is $\epsilon_\Delta>0$ such that $\Sub(H_\Delta) \cap B(\Delta,\epsilon_\Delta)=\emptyset$. Since $\Sub(K)$ is closed, the set $\Sub(G)\setminus \Sub(K)$ is Lindel\"of. So the cover $\bigcup_{\Delta\nsubseteq K} B(\Delta,\epsilon_\Delta)$ of $\Sub(G)\setminus \Sub(K)$ has a countable subcover $\bigcup_{n\in \N} B(\Delta_n,\epsilon_{\Delta,n})$. It follows that $\mu(\Sub(K))=\mu(\bigcap_{n\in \N} \Sub(H_{{\Delta}_n}))=1$.
\end{proof}

Unlike for the kernel, we do not have a good elementwise description of the closure. But we can characterise a nice set of generators.

\begin{definition} \label{def:essential}
For $\mu \in \Prob(\Sub(G))$, we denote 
\[Q(\mu) \coloneqq \{g \in G \ \mid \ \mu\left(O_1(B(g,\epsilon)\right) > 0 , \ \forall \epsilon > 0\}.\] An element $g \in Q(\mu)$ will be called a \emph{$\mu$-essential element}.
\end{definition}

\begin{lemma}\label{lem:closuregenerators}
The set of $\mu$-essential elements generates (topologically) $\langle \mu \rangle$, namely  $\langle \mu \rangle = \overline{\langle Q(\mu)\rangle}$.
\end{lemma}

\begin{proof}
For any $g \in  Q(\mu)$ and any $\epsilon >0$, the sets $O_1(B(g,\epsilon))$ and $\Sub(\langle \mu \rangle)$ intersect, since $\mu(O_1(B(g,\epsilon)))>0$ and ${\mu(\Sub(\langle \mu \rangle))=1}$. So for every $n\in \N$ there is $g_n \in \langle \mu \rangle \cap B(g,1/n)$. Since $\langle \mu \rangle$ is a closed group, this implies that $Q(\mu) \subset \langle \mu \rangle$. Conversely, by definition for every $g \not \in \overline{\langle Q(\mu) \rangle}$ there is some $\epsilon>0$ such that $\mu(O_1(B(g,\epsilon)))=0$. Appealing again to the Lindel\"of  property, we obtain a countable cover $G \setminus \overline{\langle Q(\mu) \rangle} = \cup_{i \in \N} B(g_i,\epsilon_i)$ by such balls. Consequently, \[\Sub(G) \setminus \Sub(\overline{\langle Q(\mu) \rangle}) \subset \cup_{i \in N} O_1(B(g_i,\epsilon_i)),\]
where the latter set has measure zero as a countable union of null sets. 
\end{proof}

When $G$ is a discrete group, then $Q(\mu) = \{g \in G\mid \mu(\Env g)>0\}$. In this case, \cref{lem:closuregenerators} boils down to Lemma 3.3 in \cite{Bader_Duchesne_Lecureux_Wesolek_2016}.

\begin{definition}
A probability measure $\mu\in \Prob(\Sub(G))$ is \emph{trivial} if $\langle\mu\rangle=\trivgp$ or equivalently $\mu=\delta_{\langle e\rangle}$. If on the other hand $\langle \mu \rangle=G$, then $\mu$ is called \emph{spanning} (see Definition 3.2 in \cite{Bader_Duchesne_Lecureux_Wesolek_2016}).
\end{definition}

If ${\mu\in \IRS(G)}$, then $\langle \mu \rangle$ is normal. The trivial IRS $\delta_{\langle e\rangle}$ is induced by an essentially free action.

For Dirac measures $\delta_H$, we have $\ker(\delta_H)=\langle \delta_H \rangle=H$. In that sense, the tension between $\ker(\mu)$ and $\langle \mu \rangle$ measures how far a probability measure on $\Sub(G)$ is from being a subgroup.

\subsection{Intersection of IRSs}

A prominent role in the proof of \cref{thm:acylin} is played by the co-induction operation defined by Kechris and Quorning in \cite{Kechris_Quorning_2019}. We will discuss the intersection of probability measures more generally here, which gives us a connection to the kernel of a probability measure.

\begin{definition}\label{def:intersection}
The intersection of closed subgroups defines a map $\cap:\Sub(G)\times \Sub(G)\to \Sub(G), (\Delta_1,\Delta_2)\mapsto \Delta_1\cap \Delta_2$. For $\mu,\nu\in \Prob(\Sub(G))$, we define their \emph{intersection} by $\mu\cap\nu\coloneqq \cap_*(\mu\times\nu)\in \Prob(\Sub(G))$.

Intersecting probability measures on $\Sub(G)$, we can produce new IRSs. Given $\mu\in \Prob(\Sub(G))$ and $T=G/\Stab(\mu)$, we get an IRS $\bigcap_{t\in T}t_*\mu$.
\end{definition}

\begin{remark}
One can generalize this construction: Given $\mu\in \Prob(\Sub(G))$ and $T=G/\Stab(\mu)$, we define \begin{equation*}
\phi: 2^T\to \Prob(G), \Theta \mapsto \bigcap_{t\in\Theta} t_*\mu.
\end{equation*} Let $\nu \in \Prob^G(2^T)$ be a $G$-invariant probability measure on $2^T$. Then $\operatorname{bar}(\phi_*\nu)$ is an IRS, since $\phi$ and $\operatorname{bar}$ are $G$-equivariant.

This construction is a slight generalization of the definition of \emph{intersectional IRSs} by Hartman and Yadin in \cite{Hartman_Yadin_2018}. They look at the case $\mu=\delta_K$ for $K\leq G$ such that $\Stab(\mu)=N_G(K)$ has infinite index in $G$. We will not use this more general construction.
\end{remark}
 
\begin{example}
Let $\Gamma_0 \leq \Gamma$ be countable groups. Kechris and Quorning defined a \emph{co-induction operation} $\operatorname{CIND}_{\Gamma^{0}}^{\Gamma}:\IRS(\Gamma_0)\to \IRS(\Gamma)$ in \cite{Kechris_Quorning_2019}. For any $\theta\in \IRS(\Gamma_0)$, the IRS $\operatorname{CIND}_{\Gamma_0}^\Gamma(\theta)$ is $\bigcap_{g\in \Gamma/\Gamma_0} g_*\theta$. See Remark 5.2 in \cite{Kechris_Quorning_2019} as well as the proof of \cref{prop:wm} below for the equivalence to the more standard definition used there.
\end{example}

The following proposition will be very useful for us. It is due to Kechris-Quorning \cite[Proposition 7.1]{Kechris_Quorning_2019} following Ioana \cite[Lemma 2.1]{Ioana:11}. For the convenience of the reader we repeat the proof.

\begin{proposition} \label{prop:wm}
Let $\Gamma$ be a countable group, $\Gamma_0 < \Gamma$ an infinite index subgroup and $\mu \in \IRS(\Gamma_{0})$. Set $m \coloneqq \bigcap_{\gamma \in \Gamma/\Gamma_0} \gamma_{*} \mu$. Then $m$ is a weakly mixing IRS on $\Gamma$.
\end{proposition}

\begin{proof}
By \cite[Proposition 13]{Abert_Glasner_Virag_2014} we can find a measure preserving action $\alpha: \Gamma_0 \times (Y,\nu) \rightarrow (Y,\nu)$ on an atomless probability space $(Y,\nu)$ such that the IRS is given by $\mu = \Stab_{*}(\nu)$.

Let $T = \{t_1,t_2,\ldots\}$ be a transversal for $\Gamma_0$ in $\Gamma$. For every $\gamma \in \Gamma, t \in T$ we obtain a unique decomposition
\[\gamma t = (\gamma \cdot t) \beta(\gamma,t), \quad \gamma \cdot t \in T, \ \ \beta(\gamma,t) \in \Gamma_0,\] where $(\gamma,t)\mapsto \gamma \cdot t$ is the natural action of $\Gamma$ on $T$ and $\beta:\Gamma \times T \rightarrow \Gamma_0$ is a cocycle in the sense that $\beta(\gamma_1 \gamma_2,t) = \beta(\gamma_1,\gamma_2\cdot t)\beta(\gamma_2,t)$ for any $\gamma_1,\gamma_2 \in \Gamma, t \in T$.
We think of $\Xi= T \times (Y,\nu)$ as a bundle over $T$ with fiber $(Y,\nu)$. The cocycle $\beta$ gives rise to an action $\alpha_1: \Gamma \times \Xi \rightarrow \Xi$ given by the formula 
\[\alpha_1(\gamma) (t,y) = (\gamma \cdot t, \alpha(\beta(\gamma,t)) y).\]
When $T$ is finite\footnote{Or, more generally in the topological case, when $\Gamma/\Gamma_0$ admits a $\Gamma$-invariant probability measure.}, $\Xi$ is a probability space and the corresponding IRS on $\Gamma$ is isomorphic to $\int_{T} t_{*} \mu dt$ and known as the induced IRS. 

The co-induced action $\tilde{\alpha}:\Gamma \times \Omega \rightarrow \Omega$ introduced by Ioana is the natural action of $\Gamma$ on the sections $\Omega=(Y,\nu)^T = (Y^T,\nu^T)$ of the bundle $\Xi$ endowed with the product measure. This is a probability space even when $T$ is infinite. The co-induced action is given by the formula 
\[(\tilde{\alpha} (\gamma) \omega)(t) = \alpha(\beta(\gamma,\gamma^{-1}\cdot t))\left( \omega(\gamma^{-1}\cdot t)\right) = \alpha(\beta(\gamma^{-1},t)^{-1})\left( \omega(\gamma^{-1}\cdot t)\right).\]
Geometrically, this is $\Gamma$ acting on sections, by translating their ``graphs''. Let $\hat{m} = \mathrm{Coind}_{\Gamma_0}^{\Gamma} \mu := \Stab_*(\nu^T)$ be the associated IRS.

 We now show that $m = \hat{m}$, which Kechris-Quorning proved in Theorem 5.1. and Remark 5.2.  Recall that $\Core_{\Gamma}(\Gamma_0) = \bigcap_{\gamma \in \Gamma} \gamma \Gamma_0 \gamma^{-1}$ is the kernel of the action $\Gamma \curvearrowright T$. Since $\mu \in \IRS(\Gamma_{0})$, it is easy to see that $\langle m \rangle < \bigcap_{\gamma \in \Gamma/\Gamma_0} \gamma \langle m \rangle \gamma^{-1}<\Core_{\Gamma}(\Gamma_0)$ (see \cref{lem:kernelclosure}). On the other hand, if $\gamma^{-1} \cdot t \ne t$ for some $t \in T$, the following  probability vanishes
\[\nu^T (\left\{\omega\in \Omega \mid \omega(t) = \alpha(\beta(\gamma^{-1},t))\left(\omega(\gamma^{-1}\cdot t)\right)\right\})= 0,\]
since $(Y,\nu)$ has no atoms. So $\langle \hat{m} \rangle < \Core_{\Gamma}(\Gamma_0)$ as well. 
Now in order to show that $m = \hat{m}$ it is enough to show that $m(\Env F) = \hat{m}(\Env F)$ for any finite set $F \subset \Core_{\Gamma}(\Gamma_0)$, since sets of the form $\Env F$ generate the $\sigma$-algebra on $\Sub(\Gamma)$. Note that $\beta(\cdot, t):\Core_{\Gamma}(\Gamma_0) \rightarrow \Core_{\Gamma}(\Gamma_0)$ is an homomorphism given by $\beta(\gamma,t) = t^{-1}\gamma t, \ \forall t \in T$ by the definition of the cocycle. Given any finite set $F \subset \Core_{\Gamma}(\Gamma_0)$ we have
\begin{align*}
\hat{m}(\Env F)& =\nu^T(\Fix_{\tilde{\alpha}}(F)) = \prod_{t \in T} \nu \left(\Fix_{\alpha}(\beta(F^{-1},t)\right)\\
&= \prod_{t \in T} \nu (\Fix_{\alpha}(t^{-1}F^{-1}t)) 
=\prod_{t \in T} t_{*}\mu (\Env(F^{-1}))\\
&= \prod_{t \in T} t_{*}\mu (\Env F) = m(\Env F).
\end{align*} This proves $m=\hat{m}$ is the IRS given by the co-induced action.

To establish weak mixing we show that given $f_1,f_2 \ldots f_n \in L^2_0(\Omega,\nu^T)$ and $\epsilon > 0$ one can find $\gamma \in \Gamma$ such that ${\abs{\langle \gamma f_i,f_j \rangle} < \epsilon}, \ \forall 1 \le i,j \le n$ (\cite[Theorem 4.1]{BR_mixing}). Note that $L^2(Y^T,\nu^T) = \otimes_{t \in T} L^2(Y,\nu)$ is a Hilbert tensor product, so it is enough to check functions of the form $f_i = f_{i,1} \otimes f_{i,2} \otimes \ldots \otimes f_{i,N}$ with $\norm{f_i} = \norm{f_{i,1}}\norm{f_{i,2}} \ldots \norm{f_{i,N}} < \frac{\epsilon}{2}$, for some finite $N \in \N$ and $f_{i,j} \in L^2(Y,\nu)$. Since $[\Gamma:\Gamma_0] = \infty$ we can find $\gamma \in \Gamma$ such that $\{t_1,t_2,\ldots, t_N\} \cap \{\gamma \cdot t_1, \ldots, \gamma \cdot t_N\} = \emptyset$. But then for every $1 \le i,j \le N$ the functions $\gamma f_i$ and $f_j$ are independent, as they are supported on disjoint set of variables, and 
$$\abs{\langle \gamma f_i, f_j \rangle} = \norm{\gamma f_i} \norm{f_j} = \norm{f_i} \norm {f_j} \le \epsilon^2,$$ 
as required.
\end{proof}

Instead of the intersection, one can also use the group generated by subgroups (see Remark 5.2.(3) in \cite{Kechris_Quorning_2019}).

\begin{definition}\label{def:generation}
The group generated by subgroups defines a map $\langle\rangle:\Sub(G)\times \Sub(G)\to \Sub(G), (\Delta_1,\Delta_2)\mapsto \langle\Delta_1,\Delta_2\rangle$. For $\mu,\nu\in \Prob(\Sub(G))$, we define $\langle\mu,\nu\rangle\coloneqq \langle\rangle_*(\mu\times\nu)\in \Prob(\Sub(G))$.
\end{definition}

Using $\langle\rangle$ instead of $\cap$, one can produce new IRSs analogously. We will use both operations now to give another description of the kernel and the closure of a probability measure on $\Sub(G)$.

\begin{lemma}
For every $\mu \in \Prob(\Sub(G))$, we have $\delta_{\ker(\mu)}=\bigcap_{\N} \mu$ and $\delta_{\langle \mu \rangle}=\langle \rangle_\N \mu$
\end{lemma}

\begin{proof}
We fix a compatible metric on $G$ and, for the sake of this proof only, let us denote by \[O_{g,\epsilon}\coloneqq O_{1}(B(g,\epsilon)) = \{\Delta \in \Sub(G)\mid \Delta \cap B(g,\epsilon) \neq \emptyset\}.\] Let us set also $m\coloneqq \bigcap_{\N}\mu$, and $s\coloneqq\langle \rangle_{\N}\mu$. 

By definition, we have $\{\ker(\mu)\} = \Env \ker(\mu) \cap \Sub(\ker(\mu))$. Since $m(\Env \ker(\mu))=\prod_{\N} \mu(\Env \ker(\mu))=1$, it is enough to show that $m(\Sub(\ker(\mu)))=1$. If $\mu(O_{g,\epsilon}) = 1, \ \forall \epsilon > 0$ then a $\mu$-random (closed) subgroup $\Delta$ almost surely contains a convergent sequence $g_n \rightarrow g$, so $g \in \ker(\mu)$ by \cref{lem:kernel}. Thus for $g \not \in \ker(\mu)$ and $\epsilon>0$ small enough, we have $\mu(O_{g,\epsilon}) < 1$ and hence $m(O_{g,\epsilon}) < \prod_{\N}\mu(O_{g,\epsilon}) = 0$. Let $\{B_1,B_2,\ldots\}$ be a countable cover of $G \setminus \ker(\mu)$ by open balls of the form $B_i = B(h_i,\epsilon_i)$ with $m(O_{h_i,\epsilon_i})=0$. Then $\{O_{h_1,\epsilon_1},O_{h_2,\epsilon_2},\ldots\}$ is a countable cover of $\Sub(G) \setminus \Sub(\ker(\mu))$ and $m(\Sub(\ker(\mu)))=1-m(\bigcup_i O_{h_i,\epsilon_i})=1$, which proves the first claim.

The second claim is proved similarly. By definition, it is enough to show $s(\Env \langle \mu \rangle)=1$. If $g \in Q(\mu)$ is a $\mu$-essential element, then by definition $\mu(O_{g,\epsilon})>0, \ \forall \epsilon >0$. If $\{\Delta_i\}$ are countably many independent $\mu$-random subgroups, then almost surely one of them intersects $B(g,\epsilon)$ and so $s(O_{g,\epsilon}) = 1$. Since this holds for any $\epsilon>0$, the same converging sequence argument shows that $s(\Env g)=1, \forall g \in Q(\mu)$, and hence the same holds for every $g \in \langle Q(\mu) \rangle$. By \cref{lem:closuregenerators} we find a countable set of elements $\{g_1,g_2 \ldots \} \subset \langle Q(\mu) \rangle$, which is dense in $\langle \mu \rangle$. We get $s(\Env \langle \mu \rangle)=s(\bigcup_i \Env g_i)=1$, which proves the second claim.
\end{proof}

We now discuss the kernel and closure of combinations of probability measures on $\Sub(G)$ for later use.

\begin{lemma}\label{lem:kernelclosure}
Given two probability measures $\mu_1, \mu_2$ on $\Sub(G)$, the convex combination $\nu=\alpha\mu_1+(1-\alpha)\mu_2$ is again a probability measure on $\Sub(G)$ and \begin{align*}
\ker(\nu)&=\ker(\mu_1)\cap \ker(\mu_2)\\
\langle \nu \rangle&=\langle\langle \mu_1 \rangle,\langle \mu_2 \rangle\rangle
\end{align*}Given a family $(\mu_i)_{i\in I}\in \Prob(\Sub(G))$ we have \begin{align*}
\ker(\bigcap_{i\in I}\mu_i)&=\bigcap_{i\in I} \ker(\mu_i),\\
\llangle \mu_i \rangle_{i\in I} \rangle&=\llangle \mu_i \rrangle_{i\in I}.
\end{align*} We can only say that
\begin{align*}
\ker(\langle \mu_i \rangle_{i\in I})&\geq \langle \ker(\mu_i) \rangle_{i\in I},\\
\langle \bigcap_{i\in I}\mu_i \rangle&\leq  \bigcap_{i\in I} \langle \mu_i \rangle.
\end{align*}
\end{lemma}

\begin{proof}
The statements about kernel and closure of convex combinations of probability measures on $\Sub(G)$ are direct consequences of \cref{lem:kernel} and \cref{lem:closuregenerators}, respectively.

Note that $\bigcap_{i\in I}^{-1}\Env H=\prod_{i\in I}\Env H$. So for any $H\leq G$ we have $\bigcap_{i\in I}\mu_i(\Env H)=\prod_{i\in I}\mu_i(\Env H)=1$ if and only if $H\leq \ker(\mu_i)$ for all $i\in I$, that is if and only if $H\leq \bigcap_{i\in I} \ker(\mu_i)$. So $\ker(\bigcap_{i\in I}\mu_i)=\bigcap_{i\in I} \ker(\mu_i)$.

Note that $\langle \rangle_{i\in I}^{-1}\Sub H=\prod_{i\in I}\Sub H$. So for any $H\leq G$ we have $\langle \mu_i \rangle_{i\in I}(\Sub H)=\prod_{i\in I}\mu_i(\Sub H)=1$ if and only if $H\geq \langle \mu_i \rangle$ for all $i\in I$, that is if and only if $H\geq \llangle \mu_i \rrangle_{i\in I}$. So $\llangle \mu_i \rangle_{i\in I} \rangle=\llangle \mu_i \rrangle_{i\in I}$.

Note that $\Env \langle H_i \rangle_{i\in I} \supseteq \langle \Env H_i \rangle_{i\in I}$. We have \begin{align*}
\langle \mu_i \rangle_{i\in I}(\Env \langle \ker(\mu_i) \rangle_{i\in I})&\geq \langle \mu_i \rangle_{i\in I}(\langle\Env \ker(\mu_i)\rangle_{i\in I})\\
&\geq \prod_{i\in I}\mu_i(\prod_{i\in I} \Env \ker(\mu_i))=1.
\end{align*} So $\ker(\langle \mu_i \rangle_{i\in I})\geq \langle \ker(\mu_i) \rangle_{i\in I}$.

Note that $\Sub \bigcap_{i\in I} H_i = \bigcap_{i\in I} \Sub H_i$. We have \begin{align*}\bigcap_{i\in I}\mu_i(\Sub \bigcap_{i\in I}\langle \mu_i \rangle)&=\bigcap_{i\in I}\mu_i(\bigcap_{i\in I} \Sub \langle \mu_i \rangle)\\
&\geq \prod_{i\in I}\mu_i(\prod_{i\in I} \Sub \langle \mu_i \rangle)=1.\end{align*} So $\langle \bigcap_{i\in I}\mu_i \rangle\leq  \bigcap_{i\in I} \langle \mu_i \rangle$.
\end{proof}

\section{Acylindrically hyperbolic groups}\label{sec:acylin}

In this section we prove our theorems. The main ingredients are a construction of Kechris and Quorning (in \cite{Kechris_Quorning_2019}) and a theorem of Sun (in \cite{Sun_2020}). We first reproduce the needed results of Sun.

\begin{definition}[Definition 2.2 in \cite{Sun_2020}]
Let $G$ be a group with a subgroup $H$ and let $N\lhd H$. We say that the triple $(G,H,N)$ has the \emph{Cohen-Lyndon property} if there exists a left transversal $T$ of $H\llangle N\rrangle_G$ in $G$ such that $\llangle N\rrangle_G$ is the free product of its subgroups $N^t=tNt^{-1}$ for $t\in T$, denoted as \begin{equation*}
\llangle N\rrangle_G = \prod_{t\in T}^{*} N^{t}.
\end{equation*}
\end{definition}

This property was first studied by Cohen and Lyndon in \cite{LC:free_bases}. They established this property for $(G,H,N)$, where $G = F_r$ is a free group, $N = \langle n \rangle < G$ is any nontrivial cyclic subgroup and $H = E(n)$ is the maximal cyclic subgroup containing $n$. We will use an outstanding generalization of their theorem due to Sun.

\begin{theorem}[Corollary 2.7 in \cite{Sun_2020}]\label{cor:sun}
Let $G$ be an acylindrically hyperbolic group, and let $g\in G$ be a generalized loxodromic element. Then there is a unique maximal virtually cyclic subgroup $g \in E(g) < G$ containing $g$ and the triplet $(G,E(g),N)$ has the Cohen-Lyndon property for all $N\lhd E(g)$ with $N\cap F=\emptyset$ for a fixed finite set $F\subset E(g)\setminus \{1\}$.
\end{theorem}

Note that we can always find such $N\lhd E(g)$: since $E(g)$ is virtually cyclic, for every finite subset $F \subset E(g)$ there is a normal finite index subgroup $N \lhd E(g)$ with $N \cap F=\emptyset$.

We refer the reader to \cite{Osin_2016} for the definition of acylindrically hyperbolic groups and  generalized loxodromic elements. Here instead let us mention a few outstanding examples for the above situation:
\begin{itemize}
\item $G$ is a hyperbolic group and $g \in G$ is a loxodromic element for the action on the Cayley graph of $G$. 
\item $G = \Mod(\Sigma)$ is the mapping class group of a closed orientable surface of genus at least 2 and $g \in G$ is pseudo-Anosov, \cite{Bowditch_tight,PS:acyl_mcg}.
\item $G = \Out(F_n)$ and $g \in G$ is fully irreducible, \cite{BF_hyp_out}. 
\item $G = \Aut(F_n)$ or more generally $G = \Aut(\Gamma)$ where $\Gamma$ is a non-elementary hyperbolic group, \cite{Genevois_2019,Genevois_Horbez_2021}. These examples will be especially important for us later and we will discuss them in more detail. 
\end{itemize}


We now adapt the construction made by Kechris and Quorning in the proof of Proposition 8.1 in \cite{Kechris_Quorning_2019}. Let $(G,H,N)$ have the Cohen-Lyndon property. Set $F = \llangle N \rrangle_G$. Assume that $[G:HF]=\infty$ and that ${[HF:F]=m<\infty}$. We enumerate the transversal $T=\{t_1,t_2,\dots \}$ for the left cosets in $G/HF$ provided by the Cohen-Lyndon property of $(G,H,N)$. Define $\Gamma_k\coloneqq \llangle t_iN t_i^{-1}\mid i\geq k\rrangle_{F}$ and $\mu_F\coloneqq \sum_{k=1}^\infty 2^{-k} \delta_{\Gamma_k}$. Let $S=\{s_1,s_2, \dots, s_m\}\subset H$ be a transversal for the left cosets in $HF/F$ and define ${\mu_{HF}\coloneqq \frac{1}{m}\sum_{j=1}^m {s_j}_*\mu_F}$. Finally define ${\mu_G\coloneqq \bigcap_{i\geq 1} {t_i}_*\mu_{HF}}$.

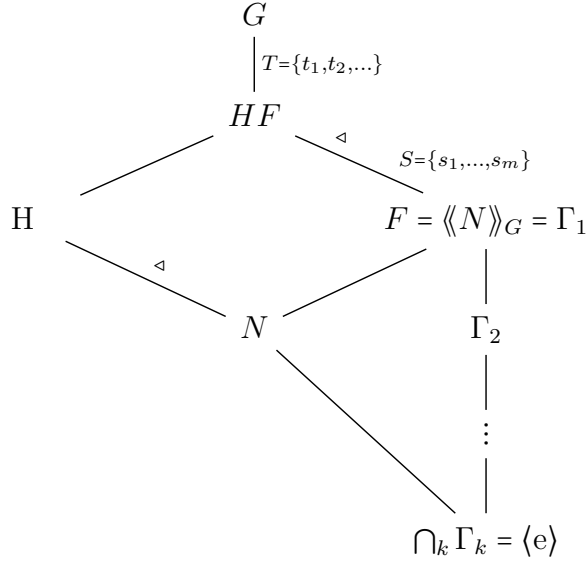
\newlength{\lena}
\settowidth{\lena}{$F=\llangle N \rrangle_G= \Gamma_1$}

\begin{center}
\begin{tikzcd}
& G \arrow[d,dash,"{T=\{t_1,t_2,\ldots\}}"] & \\
& HF \arrow[dl,dash] \arrow[dr,dash,"\lhd" near start, "{S=\{s_1,\ldots,s_m\}}" near end] & \\
\makebox[\lena][c]{H} \arrow[dr,dash,"\lhd"] & & F=\llangle N \rrangle_G= \Gamma_1 \arrow[dl,dash] \arrow[d,dash] \\
& N \arrow[ddr,dash] & \Gamma_2 \arrow[d,dash] \\
& & \vdots \arrow[d,dash] \\
&   & \bigcap_{k} \Gamma_k = \trivgp
\end{tikzcd}
\end{center}

\begin{proposition}\label{prop:CLprop}
With all of the above notation, $\mu_G$ is a weakly mixing faithful IRS on $G$ and $\langle \mu_G \rangle = F$.
\end{proposition}

\begin{proof}
By definition $\mu_F$ is an IRS on $F$, which is induced to an IRS $\mu_{HF}$ on $HF$ via the finite transversal $S$. Now, $\mu_{G}$ is obtained by co-inducing from $HF$ to $G$ and since ${[G:HF]=\infty}$, \cref{prop:wm} shows that $\mu_G$ is a weakly mixing IRS on $G$.

We will now show that $\langle \mu_G \rangle=F$. In particular, this shows that $\mu_G$ is nontrivial. For any $n\in N$, we have \begin{align*}
\mu_G(\Env n)&=\left(\bigcap_{i\geq 1} {t_i}_*\mu_{HF}\right)(\Env n)=
\prod_{i \ge 1} {t_i}_{*} \mu_{HF} \left(\Env n \right) \\
&=\prod_{i\geq 1}(\frac{1}{m}\sum_{j=1}^m\mu_F(\Env {t_i}{s_j}n{s_j}^{-1}{t_i}^{-1}))\\
\intertext{If $i\geq k$, then ${t_i}n{t_i}^{-1}\in \Gamma_k$ for all $n\in N$. Since $N\lhd H$ and $s_i \in H$, this implies $\delta_{\Gamma_k}(\Env {t_i}{s_j}n{s_j}^{-1}{t_i}^{-1})=1$ for $i\geq k$. So}
\mu_{G}(\Env n)&\geq\prod_{i\geq 1}(\frac{1}{m}\sum_{j=1}^m\sum_{k=1}^i 2^{-k}) \geq\prod_{i\geq 1}(1-2^{-i})>0.
\end{align*} We have seen that $N\subset \llangle\{g\in G\mid \mu_G(\Env g)>0\}\rrangle_G=\langle \mu_G\rangle$. Since $\langle \mu_G\rangle$ is a normal subgroup of $G$, this shows $F = \llangle N \rrangle_{G} \subset\langle \mu_G\rangle$. On the other hand, using \cref{lem:kernelclosure} we can see that \begin{align*}\langle \mu_G\rangle\leq  \bigcap_{i\geq 1} \langle {t_i}_*\mu_{HF}\rangle&=\bigcap_{i\geq 1} t_i \left\langle \langle {s_1}_*\mu_F\rangle, \dots, \langle {s_m}_* \mu_F\rangle \right \rangle t_i^{-1} \\
&=\bigcap_{i\geq 1} t_i \langle s_1Fs_1^{-1}, \dots, s_mFs_m^{-1} \rangle t_1^{-1} \leq F.\end{align*}

Now we show that $\mu_G$ is faithful. Using \cref{lem:kernelclosure} we get that $\ker(\mu_G)=\bigcap_{i\geq 1} {t_i}\ker(\mu_{HF}){t_i}^{-1}$ and $\ker(\mu_{HF})=\bigcap_{j=1}^m {s_j}\ker(\mu_F){s_j}^{-1}$. 
So it is enough to show that $\ker(\mu_{F})=\bigcap_{k=1}^\infty \Gamma_k$ is trivial. Remember that \[\Gamma_k=\llangle t_iN t_i^{-1}\mid i\geq k\rrangle_{F}<F=\prod_{i\geq 1} t_iNt_i^{-1}.\] Define $\phi_k:F\to F$ by \[\phi_k(t_in t_i^{-1})=\begin{cases}t_in t_i^{-1},& \text{if } i<k\\
e,& \text{if } i\geq k\\
\end{cases}\] for any $n\in N$. Note that $\Gamma_k=\ker(\phi_k)$. On the other hand, for any nontrivial element $\omega\neq e \in F$, there is $l\in \N$ such that $\phi_l(\omega)=\omega\neq e$. So $\bigcap_{k=1}^\infty \Gamma_k=\langle e\rangle$.
\end{proof}

\acylin

\begin{proof}
Let $G$ be an acylindrically hyperbolic group, and let $g\in G$ be a generalized loxodromic element. If $G$ is virtually free, with a normal finite index free subgroup $F \lhd G$, we further require that $g$ be chosen as a loxodromic element inside the commutator $F'$. By \cref{cor:sun}, $(G,E(g),N)$ has the Cohen-Lyndon property for some finite index cyclic subgroup $N\lhd E(g)$. Since $[E(g):N]<\infty$, we also have $[E(g)\llangle N\rrangle_G:\llangle N\rrangle_G]<\infty$. If $G$ is virtually free, then our choice of $g$ ensures that $\llangle N \rrangle_G < F'$ which is of infinite index in $G$. If $G$ is not virtually free, $\llangle N \rrangle_G$, which is a free group due to the Cohen-Lyndon property, again must be of infinite index. In any case ${[G:E(g)\llangle N\rrangle_G]=\infty}$, so we can apply \cref{prop:CLprop} and conclude that $G$ admits a weakly mixing nontrivial faithful IRS.
\end{proof}

\begin{remark}
We can find a continuum of ergodic nontrivial faithful IRSs in every acylindrically hyperbolic group $G$, adapting the construction of $\mu_G$ exactly as in the proof of Proposition 8.1 in \cite{Kechris_Quorning_2019}. Let $N\in \N$ be minimal, such that $\Gamma_{N+1}\subsetneq \Gamma_N$ and let $\lambda=\sum_{k=1}^{N+1} 2^{-k}$. For $a\in (0,\lambda)$, define \[\mu_{a,F}=a\delta_{\Gamma_1}+(\lambda-a)\delta_{\Gamma_{N+1}}+\sum_{k=N+2}^\infty 2^{-k} \delta_{\Gamma_k}.\] Then $\mu_{a,F}$ is an IRS on $F$, which can be induced to $HF$ and then co-induced to an ergodic nontrivial faithful IRS $\mu_{a,G}$ on $G$ as above. Choose $n\in N$ such that $A(n)\coloneqq \{i\in \N\mid t_int_i\notin \Gamma_{N+1}\}$ is not empty. Then \[
\mu_{a,G}(\Env n)=a^{\lvert A(n)\rvert} \prod_{i\notin A(n)}(\frac{1}{m}\sum_{j=1}^m\mu_{a,F}(\Env {t_i}{s_j}n{s_j}^{-1}{t_i}^{-1})),\] so $\mu_{a,G}$ does depend on the choice of $a$.
\end{remark}

\charirs
\begin{proof}
The results of Genevois (in \cite{Genevois_2019}) and Genevois-Horbez (in \cite{Genevois_Horbez_2021}) together show that $\Aut(G)$ is acylindrically hyperbolic if $G$ is a non-elementary hyperbolic group. They have also shown that there is an inner automorphism $c_g\in \Inn(G)\lhd \Aut(G)$ that is a generalized loxodromic element. So we get not only a weakly mixing nontrivial faithful IRS on $\Aut(G)$ by \cref{prop:CLprop}, but this IRS is contained in $\Inn(G)$. Therefore we can see it as a characteristic random subgroup of $G$.
\end{proof}

\begin{remark}\label{rem:induction}
Using induction, we can conclude that every group containing a non-elementary hyperbolic group as a normal subgroup, also admits an weakly mixing nontrivial faithful IRS. If $G$ admits a finitely supported IRS $\mu$ such that ${P_\mu(H\cong\Gamma)=1}$ for some non-elementary hyperbolic group $\Gamma$, we can still conclude that $G$ admits an ergodic nontrivial faithful IRS.
\end{remark}

\section{Some open questions}

We answered the weak part of \cref{q:1} for acylindrically hyperbolic groups, namely we provided weakly mixing nontrivial IRSs with trivial kernel. But the IRSs we constructed are not spanning. In fact, in our construction the closure of the IRS is always a countably generated free group
\[F_\infty \cong \langle \mu \rangle = \prod^{*} t_i N t_i^{-1}, \qquad N \cong \Z,\]
which is bound to be of infinite index inside the ambient acylindrically hyperbolic group. So we remain with the following harder version of our original question: 
\begin{question}\label{q:2}
Which (countable) groups admit ergodic faithful spanning IRSs?
\end{question}

In \cref{thm:charirs} we construct a weakly mixing nontrivial faithful characteristic random subgroup inside any non-elementary hyperbolic group. A very natural group that is left out here is the group $F_\infty$, which is not hyperbolic and not even acylindrically hyperbolic. The main result of \cite{Bowen_Grigorchuk_Kravchenko_2017} does construct a weakly mixing nontrivial characteristic random subgroup in $F_\infty$, which is far from being faithful. Thus we are left with the question:
\begin{question} \label{q:3}
Does the group $F_\infty$ admit an ergodic nontrivial faithful characteristic random subgroup? How about the fundamental group $\pi_1(\Sigma_{\infty})$ of a surface of infinite genus? 

\end{question}
\bibliography{LitFreevsFaithful}{}
\bibliographystyle{abbrv}

\bigskip

\noindent Both authors were partially funded by Israel Science Foundation grant ISF 2919/19.
\end{document}